\documentclass[a4paper,reqno]{amsart}

\usepackage{latexsym}
\usepackage[english]{babel}
\usepackage{fancyhdr}
\usepackage[mathscr]{eucal}
\usepackage{amsmath}
\usepackage{mathrsfs}
\usepackage{mathabx}
\usepackage{amsthm}
\usepackage{amsfonts}
\usepackage{amssymb}
\usepackage{amscd}
\usepackage{bbm}
\usepackage{graphicx}
\usepackage{graphics}
\usepackage{latexsym}
\usepackage{color}
\usepackage{subfig}
\usepackage{hyperref}

\usepackage{scalerel,stackengine}
\stackMath
\newcommand\reallywidehat[1]{%
\savestack{\tmpbox}{\stretchto{%
  \scaleto{%
    \scalerel*[\widthof{\ensuremath{#1}}]{\kern-.6pt\bigwedge\kern-.6pt}%
    {\rule[-\textheight/2]{1ex}{\textheight}}
  }{\textheight}%
}{0.5ex}}%
\stackon[1pt]{#1}{\tmpbox}%
}

\theoremstyle{plain}
\newtheorem{theorem}{Theorem}
\newtheorem{lemma}[theorem]{Lemma}

\newtheorem*{thm}{Theorem}

\theoremstyle{definition}

\newtheorem*{remark*}{Remark}

\numberwithin{equation}{section}

\begin{document}
\title[]{Guessing cards with complete feedback}

\address[]{Department of Mathematics, University of Washington, Seattle, WA 98195}

\author[]{Andrea Ottolini}
\email{ottolini@uw.edu}

\author[]{Stefan Steinerberger}
\email{steinerb@uw.edu}

\thanks{S.S. was partially supported by the NSF (DMS-2123224) and the Alfred P. Sloan Foundation.}

\subjclass[2010]{62L99, 60G40, 92A25} 

\begin{abstract}
We consider the following game that has been used as a way of testing claims of extrasensory perception (ESP). One is given a deck of $mn$ cards comprised of $n$ distinct types each of which appears exactly $m$ times: this deck is shuffled and then cards are discarded from the deck one at a time from top to bottom. At each step, a player (whose psychic powers are being tested) tries to guess the type of the card currently on top, which is then revealed to the player before being discarded. We study the expected number $S_{n,m}$ of correct predictions a player can make: one could always guess the exact same type of card which shows that one can achieve $S_{n,m} \geq m$. We prove that the optimal (non-psychic) strategy is just slightly better than that and
$$ S_{n,m} = m+\frac{\pi}{\sqrt{2}}\sqrt{m\ln n} + o(\sqrt{m \ln{n}})$$
 whenever $(\ln n)^{3+\varepsilon}\ll  m$. This is very different from the case where $m$ is fixed and $n \rightarrow \infty$ (He \& Ottolini) and similar to the case of fixed $n$ and $m \rightarrow \infty$ (Graham \& Diaconis). The case $m=n$ answers a question of Diaconis. 
\end{abstract}

\maketitle

\section{Introduction}
\subsection{Zener cards} Sometimes people present claims of having powers of extrasensory perceptions: a natural framework (proposed by the psychologist K. Zener and the botanist J. Rhine) in which to test such a hypothesis is that of card guessing (so-called Zener cards).
Consider a well-mixed deck of $mn$ cards comprised of $n$ distinct types of cards each of which appears exactly $m$ times. The deck is well shuffled and then placed in front of the player (who has full knowledge of the composition of the deck). The player then has to guess the type of the card on top; after the guess is made, the card is shown to the player and then discarded from the deck. The game continues until the deck runs over. Assuming the player does \textit{not} have psychic abilities, how many correct guesses can one expect?

\begin{center}
\begin{figure}[h!]
    \centering
    \includegraphics[width=0.5\textwidth]{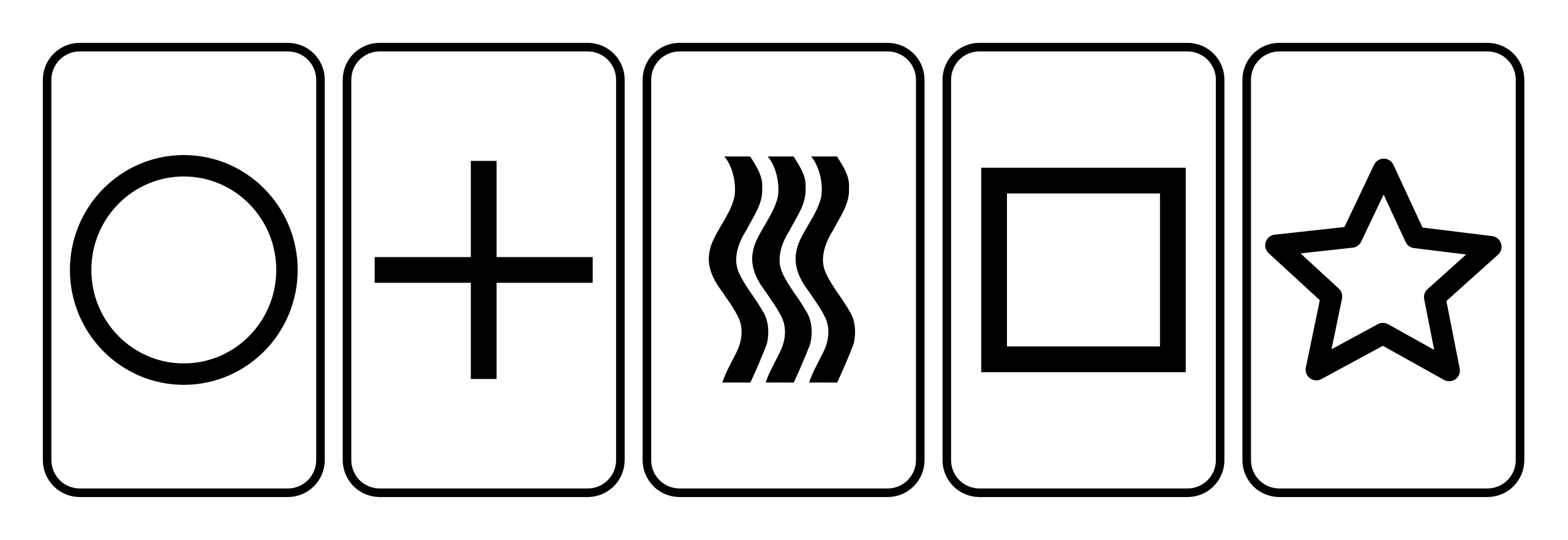}
    \caption{Zener cards: the figure shows the $n=5$ different types each of which appears $m=5$ times for a total of 25 cards.}
    \label{fig:my_label}
\end{figure}
\end{center}

\vspace{-15pt}

A simple strategy would be to always guess the same type (say, the circle card). One is guaranteed to make exactly $m$ correct guesses. The optimal strategy is to memorize all cards that have been discarded up to now and guess the type of card that has, so far, appeared the fewest amounts of times (the optimality of this strategy was proven by Diaconis \& Graham \cite{DG81}). The natural question is now, assuming no powers of extrasensory perception,  how many correct guesses can be expected under this optimal strategy? 
This game has also been analyzed in connection with clinical trials \cite{Blackwell1957, EFRON1971} and is generally well studied: for other results and variations, we refer to 
\cite{C98,Diaconis1978,DG81,diaconis2020card,diaconis2020guessing,KP01,KT21,KPP09,L21,P09,P91,S21}.

\subsection{Results.} For any given $m,n$, we consider the quantity $S_{n,m}$ describing the expected number of correctly guessed cards under the optimal strategy.
Two types of regimes are well understood. The first regime deals with the case where $n$, the number of distinct types, is fixed and the multiplicity $m$ with which each type appears becomes larges.
\begin{thm}[Diaconis \& Graham \cite{DG81}]
 For the number of different types $n$ fixed and the multiplicity $m$ going to infinity, 
\begin{equation}\label{largemsmalln}
S_{n,m}=m+\frac{\pi}{2}M_n\sqrt{m}+o_n(\sqrt m),
\end{equation}
where $M_n$ denotes the expected value of the maximum of $n$ normal random variables.
\end{thm}

One way of interpreting the result is perhaps as follows: the frequency with which each card appears should behave roughly like a normal distribution. Exploiting the fluctuations of $n$ Gaussians and taking the one that deviates the most from its expectation suggests an asymptotic along the lines given by Diaconis \& Graham. Note that this is merely a heuristic: these card counts are not actually independent.
In the opposite regime, fixing the multiplicity $m$ and assuming the number of different types $n$ becomes large, we obtain a very different result.

\begin{thm}[He \& Ottolini \cite{he2021card}]
 For the multiplicity $m$ fixed and the number of different types $n$ going to infinity, 
\begin{equation}\label{largemsmalln}
S_{n,m}=H_mH_n+\sum_{j=1}^{m-1}\frac{1}{j}\ln{m\choose j}+O_m(n^{-1/m}),
\end{equation}
where $H_n = 1 + \dots + 1/n$ denotes the $n$-th harmonic number.
\end{thm}
The leading order term is $H_m H_n \sim \ln{n} \ln{m}$ which, for $m$ fixed, is logarithmic growth in $n$. The second term in the expansion only depends on $m$ and is thus a constant. Empirically, the result is accurate even for small values of $m,n$ (see \cite{he2021card}).\\

A natural remaining question is what happens when both $m,n \rightarrow \infty$ with the original Zener setup $m=n$ being perhaps particularly interesting. Our main result covers a wide range of these parameters and is applicable as long as the number of different cards  $n$ is slightly smaller than exponential in the number of different types $m$. Such a restriction is necessary: when $n$ becomes disproportionately large compared to $m$, the result of He \& Ottolini \cite{he2021card} shows the behavior to be different.

\begin{thm}[Main Result] \label{mainthm} Let $c, \varepsilon > 0$. If $m,n \rightarrow \infty$ while $(\ln{n})^{3+\varepsilon} \leq c \cdot m$, then
\begin{equation}\label{mnwhatever}
S_{n,m}=m+\frac{\pi}{\sqrt{2}}\sqrt{m\ln n}+o_{c,\varepsilon}(\sqrt{m\ln n}).
\end{equation}
\end{thm}

The result covers a wide range of parameters. It also suggests that there is a phase transition in the regime where there are a great many different types of cards each of which only appearing a relatively small number of times. In that regime, we expect a switch from \ref{mnwhatever} to \ref{largemsmalln}. The nature of this transition is currently not understood and appears to be an interesting problem: the proof of our main results suggests that this phasse transition may perhaps occur around $\ln{n} \sim m$. Of course, many other problems (variance or the existence of a central limit theorem) remain.
There is a heuristic that motivates our main result. We use $X_i(t)\in \{0,1,\ldots , m\}$ to denote the numbers of cards of type $1\leq i\leq n$ that are left in the deck when there are $1\leq t\leq nm$ cards left in total. Linearity of expectation and the description of the optimal strategy imply that
\begin{equation}\label{linearita}
S_{n,m}=\sum_{t=1}^{nm}\frac{\mathbb E[\max_i X_i(t)]}{t}.
\end{equation}
We rescale $p = t/nm$ (thus $0 \leq p \leq 1$). 
The $X_i(t)$ should approximately obey a normal distribution and we could moreover assume that they are independent. This is certainly false because $X_1(t) + \dots +X_n(t) = t$ but it would simplify the problem. Pretending that the $X_i(t)$'s are independent normal random variables with the correct mean $mp$ and variance $mp(1-p)$, one would obtain 
\begin{align*}
\mathbb E[\max_i X_i(t)]\approx m p+\sqrt{2mp(1-p)\ln n}.
\end{align*}
Plugging in, we obtain (after substituting $p = t/nm$)
\begin{align*}
   S_{n,m}=\sum_{t=1}^{nm}\frac{\mathbb E[\max_i X_i(t)]}{t} &\approx m + \sum_{t=1}^{mn} \frac{\sqrt{2mp(1-p)\ln n}}{t} \\
   &= m + \sqrt{2 m \ln{n}}\sum_{t=1}^{mn} \frac{\sqrt{p(1-p)}}{t} \\
   &\approx  m + \sqrt{2 m \ln{n}}\int_0^1\sqrt{\frac{1-p}{p}}dp
\end{align*}
and the integral evaluates to $\pi/2$. 
\section{Proof}
\subsection{Outline} Our proof will be motivated by the heuristic sketched above. To justify the heuristic, we will exploit a well-known conditional representation of the $X_i(t)$'s in terms of conditionally independent binomial random variables $Y_i(t)$ given their sum. The intuition is that the maximum should only be mildly affected by the conditioning on the sum, which is indeed the case. One has to be careful in the case of of small and large $t$ (having selected almost none or almost all of the cards) where approximations degenerate. We will split the argument into two parts. In Section \ref{sec2}, we show how to reduce the problem to the case where the $X_i$s are independent binomials by means of a conditional representation. In Section \ref{sec3}, we prove the result in the case of independent binomials by exploiting sharp bounds on binomials tails. These two ingredients then establish the result. Section \ref{prooof} contains the proof of the main result.

\subsection{Reduction to independence}\label{sec2} We start by explaining how to reduce the problem to that of independent random variables by means of a useful conditional representation. The use of conditional limit theory to deal with order statistics of discrete processes dates back to \cite{Levin1981}, where the author exploits a conditional representation of the multinomial distribution in terms of independent Poisson conditioned on their sum.\\ 

Recall that, for each $1 \leq i \leq n$ and each $1\leq t\leq mn$, the random variable $X_i(t)$ counts the number of cards of type $i$ that are in the remaining $t$ cards. 
Fixing a value of $t$ their joint distribution is given by a multivariate hypergeometric distribution
\begin{equation}
\mathbb P(X_1(t)=j_1,\ldots, X_n(t)=j_n)=\frac{\prod_{i=1}^n {m \choose j_i}}{{nm \choose t}}, \quad j_1+\ldots+j_n=t, \quad 0\leq j_i\leq m
\end{equation}
Note that, if it were not for the constraint of having a total of $t$ cards remaining
$$ \sum_{i=1}^{n} X_i(t) = \sum_{i=1}^{n} j_i=t,$$
the $X_i$ would be independent.
To overcome this issue, we consider $n$ independent and identically distributed random variables $Y_1(t), \ldots, Y_n(t)$ each of which follow an independent binomial distribution
$$Y_i\sim \mbox{Bin}(m,p) \quad \mbox{with}~ p=\frac{t}{mn}.$$
We also introduce their sum
$$\tilde Y(t)=\sum_{i=1}^{n} Y_i(t)$$
 and note that $\tilde Y(t)\sim \mbox{Bin}(mn,p)$.
 We will use the fact that for having $t$ fixed, the distribution of cards can be realized via independent and identical random variables following a binomial distribution and conditioned on having the correct sum (in particular, the binomial random variables are also not yet independent). This well-known characterization of the hypergeometric distribution (see, e.g., \cite{Skibinsky1970}) has a simple proof that we report here for the sake of completeness.\\  
\begin{lemma}\label{conditionalrep}
For any $n,m$ and $1\leq t\leq mn$ fixed, we have
\begin{align*}
\mathcal L(X_1(t),\ldots, X_n(t))=\mathcal L(Y_1(t)\ldots, Y_n(t)|\tilde Y(t)=t),
\end{align*}
where $\mathcal{L}$ denotes the law of the random variable.
\end{lemma}
\begin{proof}
Consider any $n$-tuple $0\leq j_i\leq m$ with $\sum_{i=1}^n j_i=t$, and let $p=t/mn$. By definition of conditional expectation
\begin{align*}
\mathbb{P}(Y_1(t)=j_1,\ldots, Y_n(t)=j_n|\tilde Y(t)=t)&=\frac{\mathbb P(Y_1(t)=j_1,\ldots, Y_n(t)=j_n)}{\mathbb P(\tilde Y(t)=t)},
\end{align*}
where the condition $\tilde Y(t) = t$ can be omitted because $\sum_{i=1}^n j_i=t$ by design. We can now use the independence of the $Y_i$ to compute
\begin{align*}
   \mathbb P(Y_1(t)=j_1,\ldots, Y_n(t)=j_n) &= \prod_{i=1}^{n} {m \choose j_i}p^{j_i}(1-p)^{m-j_i} \\
   &= p^{\sum_{i=1}^{n} j_i} (1-p)^{mn - \sum_{i=1}^{n} j_i }\prod_{i=1}^{n} {m \choose j_i} \\
   &= p^t (1-p)^{mn - t} \prod_{i=1}^{n} {m \choose j_i}.
\end{align*}
Simultaneously, since $\tilde Y(t)\sim \mbox{Bin}(mn,p)$, we have
$$ \mathbb P(\tilde Y(t)=t) = \binom{mn}{t} p^t (1-p)^{mn-t}$$
from which we deduce the desired statement.
\begin{align*}
\mathbb{P}(Y_1(t)=j_1,\ldots, Y_n(t)=j_n|\tilde Y(t)=t)=\frac{\prod_{i=1}^n {m \choose j_i}}{{nm \choose t}}.
\end{align*}
\end{proof}
We define $\tilde S_{n,m}$ to be the analogue of \eqref{linearita} where we replace the hypergeometric random variables with \textit{independent} binomial random variables, i.e.
\begin{equation} \label{eq:better}
\tilde S_{n,m}=\sum_{t=1}^{mn}\frac{\mathbb E[\max_i Y_i(t)]}{t}.
\end{equation}
Note that, according to Lemma 1, if we condition the binomial random variables on having the correct sum $t$, we recover the hypergeometric distribution exactly: the purpose of the next Lemma is to show that omitting this conditioning leads to a small error. This will the conclude the first part of the proof, the remainder of which is then dedicated to the study \ref{eq:better}.
\begin{lemma}\label{randomizing}
We have, for some universal $C>0$,
\begin{align*}
|S_{n,m}-\tilde S_{n,m}| \leq C(\sqrt m+\ln n).
\end{align*}
\end{lemma}
\begin{proof}
Owing to Lemma \eqref{conditionalrep}, we can replace the independent binomial random variables $Y_i$ by hypergeometric random variables $X_i$ provided that we condition on their sum $t$: this allows us to write
\begin{align*}
S_{n,m}-\tilde S_{n,m}&=\sum_{t=1}^m \frac{\mathbb E[\max_i X_i(t)]-\mathbb E[\max_i Y_i(t)]}{t}\\&=\sum_{t=1}^{nm}\frac{1}{t}\sum_{s=1}^{nm}\left(\mathbb E[\max_i X_i(t)]-\mathbb E[\max_i X_i(s)]\right)\mathbb P(\tilde Y(t)=s).
\end{align*}
One would of course now expect that $P(\tilde Y(t)=s)$ is tightly concentrated around its expectation: it thus suffices to understand how quickly the maximum can change for $|t-s|$ relatively small (without loss of generality, we assume from now on $s<t$).
We therefore have to understand the likelihood $\mathbb P\left(\max_i X_i(t+1)>\max_i X_i(t)\right)$. The maximum can only increase if a card is picked that is already maximal before. This suggests on conditioning on the number of types that have appeared a maximal number of times and we note that
\begin{equation}\label{excursion}
\mathbb P\left(\max_i X_i(t+1)>\max_i X_i(t)\Big||\ell: X_{\ell}(t)=\max_j X_j(t)=k|=j\right)=\frac{(m-k)j}{nm-t}\leq \frac{j}{n},
\end{equation}
where we used that
$$ t = \sum_{i=1}^{n} X_i(t) \leq n \max_{1 \leq i \leq n} X_i(t)$$
implying
$$ k = \max_{1\leq i \leq n} X_i(t) \geq\frac{t}{n}.$$
The only case in which \eqref{excursion} is saturated is the configuration in which $j$ cards appear with multiplicity $k$, and all other cards appear with multiplicity $k-1$.  In this case the maximum increases with probability precisely $j/n$.
One way of seeing this is as follows: if the other cards had only appeared rarely up to that point,  we would be more likely to pick one of these cards since there are still more of them in the pile. The most likely transition to a new maximum happens if the chance of picking a card that is already chosen a maximal number of times is greatest. \\

We will now present an argument which, implicitly, works under the assumption that we are constantly in the worst case setting described above (in a suitable sense). We  introduce a Markov chain whose role is to keep track of the number of different types of cards whose current occurence is given by maximal multiplicity $\max_{i} X_i(t)$. Note that, in particular, if the maximum increases then there exists exactly one card which arises with maximal multiplicity and the counter drops back to 1. The Markov chain will be
operating on the state space $\{1,\ldots, n\}$: it is possible to move from each point $j$ to either $j+1$ or back to $1$. The corresponding transition probabilities $q_{i,j}$ are given by
\begin{align*}
q_{j, j+1}=\frac{n-j}{n} \quad \mbox{and} \quad  q_{j,1}=\frac{j}{n}  \quad \mbox{for} \qquad 1\leq j\leq n.
\end{align*}
The further the Markov chain is from 1, the more likely it is to return to 1 (corresponding to uncovering a new maximum). We also observe that the Markov chain is more likely to return to 1 than we are to uncover a new maximum (because the card deck will not always be in the worst case scenario assumed above): more formally, for all $1\leq k\leq m$, $1\leq t\leq mn$ and $1\leq j\leq j'\leq n$
\begin{align*}
\mathbb P\left(\max_i X_i(t+1)>\max_i X_i(t)\Big||\ell: X_{\ell}(t)=\max_j X_j(t)=k|=j\right)\leq q_{j,1}\leq q_{j',1}, \quad 
\end{align*}
The number of times at which $\max X_i(t)$ changes are bounded above by the number $N_{t-s}$ of excursions away from $1$ of the Markov chain.
 In particular,  
\begin{align*}
\mathbb E[\max_i X_i(t)]-\mathbb E[\max_i X_i(s)]\leq \mathbb E[N_{t-s}].
\end{align*}
It thus remains to understand how often the Markov chain is going to hit the state 1.
Let now $T$ be the time to return back to the origin for the Markov chain $Z$. Renewal theory suggest that the latter should be approximately $(t-s)\mathbb E[T]$ for $t-s$ large. This is made precise in \cite{renewal}, whose result gives the estimate
\begin{align*}
\mathbb E[N_{t-s}]\leq \frac{t-s}{\mathbb E[T]}+O\left(\frac{\mathbb E[T^2]}{(\mathbb E[T])^2}\right).
\end{align*}
Note that $T$ is nothing but the expected time to observe a birthday coincidence in the classical birthday problem. In particular,
$$\mathbb P(T\geq s)=\left(1-\frac{s}{n}\right)\ldots \left(1-\frac{1}{n}\right)$$
and using the estimate
\begin{align*}
\exp\left(-\frac{s^2}{2n}+O\left(\frac{s^3}{n^2}\right)\right)\leq \left(1-\frac{s}{n}\right)\ldots \left(1-\frac{1}{n}\right)\leq \exp\left(-\frac{s^2}{2n}\right),
\end{align*}
we obtain the well-known results
$\mathbb E[T]=\Omega(\sqrt n)$ and $\mathbb E[T^2]=O(n)$
and thus
\begin{align*}
\mathbb E[\max_i X_i(t)]-\mathbb E[\max_i X_i(s)]=O\left(\frac{t-s}{\sqrt n}+1\right).
\end{align*}
This implies
\begin{align*}
|S_{n,m}-\tilde S_{n,m}|&=O\left(\frac{1}{\sqrt n}\sum_{t=1}^{nm}\frac{1}{t}\sum_{s=1}^{mn}\mathbb |t-s|P(\tilde Y(t)=s)+\sum_{t=1}^{nm}\frac{1}{t}\right).
\end{align*}
The second sum can be bounded by $\ln{(mn)}$. As for the first sum, we first note that by Cauchy-Schwarz for any random variable
$$ \mathbb{E} \left| X - \mathbb{E}X \right| \leq \sqrt{\mathbb{V} X}.$$
We also observe that $\tilde Y(t)\sim \mbox{Bin}(nm,p)$ with $p=t/nm$ has standard deviation $\sqrt{mnp(1-p)}$
and thus
$$ \sum_{s=1}^{mn}\mathbb |t-s|P(\tilde Y(t)=s) \leq \sqrt{m n p (1-p)}.$$
This simplifies the first sum and leads to the desired bound since
\begin{align*}
    \frac{1}{\sqrt n}\sum_{t=1}^{nm}\frac{1}{t}\sum_{s=1}^{mn}\mathbb |t-s|P(\tilde Y(t)=s) &\leq     \frac{1}{\sqrt n}\sum_{t=1}^{nm}\frac{1}{t} \sqrt{mn p (1-p)} \\
    &=\sqrt m\sum_{t=1}^{nm}\frac{1}{t} \sqrt{\frac{t}{mn} \left(1 - \frac{t}{mn}\right)} \\
    &\leq c\sqrt{m} \int_0^1 \frac{1}{x}\sqrt{x(1-x)} dx \leq C \sqrt{m}.
\end{align*}
\end{proof}

\subsection{The independent case}\label{sec3} It now suffices to analyze
$$\tilde S_{n,m}=\sum_{t=1}^{nm}\frac{\mathbb E[\max_i Y_i(t)]}{t}$$
where the $Y_i$ are \textit{independent} binomial random variables and
$$Y_i\sim \mbox{Bin}(m,p) \quad \mbox{with}~ p=\frac{t}{mn}.$$
We start by rescaling these random variables by shifting them to have mean 0: we let $\overline Y_i(t)=Y_i(t)-t/n$ which reduces our problem to the study of
\begin{equation}\label{tails}
\tilde S_{n,m}=m+\sum_{t=1}^{nm}\frac{\mathbb E[\max_i \overline Y_i(t)]}{t}.
\end{equation}
These binomial random variables are well approximated by a normal distributions in regions where their variance is not too small: this naturally suggests splitting the problem into different regions. We write, for some $ 0 < s=s_{n,m} \ll 1$ to be determined later (which will ultimately tend to 0 at a suitable rate), 
\begin{align*}
\tilde S_{n,m} =m &+ \sum_{t=snm}^{(1-s)nm}\frac{\mathbb E[\max_i \overline Y_i(t)]}{t} \\
&+\sum_{t=1}^{s nm}\frac{\mathbb E[\max_i \overline Y_i(t)]}{t} +
\sum_{t= (1-s) nm}^{nm}\frac{\mathbb E[\max_i \overline Y_i(t)]}{t}.
\end{align*}
We  treat the first sum by comparing with a normal distribution. The two remaining sums are tail events that will be treated separately. We start with the tail events.\\

\textbf{The tail sums.}
In order to deal with the tails, we just need an upper bound of the right order, but we can afford to lose a factor of $\sqrt{1-p}$. This follows from a Chernoff type argument: for all $\theta>0$, by linearity of expectation,
\begin{align*}
 \exp{\{\theta \mathbb ~\mathbb{E}[\max_i \overline Y_i(t)]\}} &\leq  \sum_{i=1}^{n} \exp{\{\theta~ \mathbb E[ \overline Y_i(t)]\}} \\
&=
n\exp{\{\theta~ \mathbb{E}[\overline Y_1(t)]\}}\leq n(1-p+pe^{\theta})^me^{-\theta m}, 
\end{align*}
so that taking logarithm of both sides and using $\ln(1+x)\leq x$ we obtain
\begin{align*}
\mathbb E[\max_i \overline Y_i(t)]\leq \frac{\ln n}{\theta}+\frac{mp}{\theta}(e^{\theta}-1)-m.
\end{align*}
We start with the case where $p$ is close to 0, which is the hardest to deal with since it is the time where the feedback is the most relevant. This allows for many extra correct guesses owing to the detailed knowledge of the composition of the deck -- captured by the logarithmic singularity close to $t=0$ in \eqref{tails}. However, we deal with that considering the choice
\begin{align*}
\theta=\sqrt{\frac{2\ln n}{m}}\left(\ln\frac{1}{p}\right)^{1+\varepsilon}
\end{align*}
Since $p = t/(mn) \geq 1/(mn)$ we deduce
$$\ln \frac{1}{p}\leq \ln n+\ln m$$ and thus
\begin{align*}
\theta=O\left(\sqrt{\frac{(\ln n+\ln m)^{3+2\varepsilon}}{m}}\right)=o(1)
\end{align*}
using the main assumption on $m$ and $n$ from the Main Theorem. Since $\theta$ is tending to 0, we can replace the exponential function $e^{\theta}$ by a second order Taylor expansion and 
\begin{align*}
\mathbb E[\max_i \overline Y_i(t)]&\leq \frac{\ln n}{\theta}+\frac{mp\theta}{2}+o(\theta)\\&=O\left(\sqrt{m\ln n}\left[\left(\ln\frac{1}{p}\right)^{-1-\varepsilon}+p\left(\ln\frac{1}{p}\right)^{1+\varepsilon}\right]\right)\\
&= O\left( \sqrt{m\ln n}\left(\ln\frac{1}{p}\right)^{-1-\varepsilon} \right).
\end{align*}
Therefore, we conclude
\begin{align*}
\sum_{t=1}^{s nm}\frac{\mathbb E[\max_i \overline Y_i(t)]}{t} &\lesssim \sqrt{m \ln{n}} \sum_{t=1}^{s m n} \frac{1}{t} \left(\ln \frac{mn}{t} \right)^{-1 - \varepsilon}.
\end{align*}
Comparing to the integral, we have
$$\sum_{t=1}^{s m n} \frac{1}{t} \left(\ln \frac{mn}{t} \right)^{-1 - \varepsilon} \lesssim \int_0^{s} \frac{1}{x} \left(\ln \frac{1}{x} \right)^{-1 - \varepsilon} dx.$$
The integrand has an antiderivative in closed form
$$ \int \frac{1}{x} \left(\ln \frac{1}{x} \right)^{-1 - \varepsilon} dx = \frac{1}{\varepsilon} \left( \ln\frac{1}{x} \right)^{-\varepsilon}$$
allowing us to deduce that
$$ \int_0^{s} \frac{1}{x} \left(\ln \frac{1}{x} \right)^{-1 - \varepsilon} dx =\frac{1}{\varepsilon} \left( \ln \frac{1}{s} \right)^{-\varepsilon}$$
which, for fixed $\varepsilon$ tends to 0 provided that $s$ tends to 0. In the second regime, $p$ close to 1, we choose 
\begin{align*}
\theta=\sqrt{\frac{2\ln n}{m}}
\end{align*}
which is guaranteed to converge to 0 as $m,n \rightarrow \infty$. A Taylor expansion and the observation $p\leq 1$ shows
\begin{align*}
\mathbb E[\max_i \overline Y_i(t)]\leq \frac{\ln n}{\theta}+\frac{mp\theta}{2}+o(\theta)=O(\sqrt{m\ln n}).
\end{align*}
From here, we conclude since
\begin{align*}
\sum_{t=(1-s)mn}^{ nm -1}\frac{\mathbb E[\max_i \overline Y_i(t)]}{t} &\lesssim \sqrt{m \ln{n}} \sum_{t=(1-s)mn}^{m n -1}\frac{1}{t} \lesssim s \sqrt{m\ln n}
\end{align*}
which again tends to 0 as $s \rightarrow 0$.
\subsection{The main term}
As for the main term, we need sharp asymptotics for the maximum of independent binomials. This amounts to controlling their tail probabilities. The following is a corollary of a more general result by Feller \cite{feller}.
\begin{lemma}[Feller \cite{feller}]\label{fellermagic}
Let $Y_{i}(t) \sim \emph{Bin}(m,p)$ be i.i.d. and $p = t/(nm)$. Let
\begin{align*}
Z_i(t)=\frac{Y_i(t)-mp}{\sqrt{mp(1-p)}}
\end{align*}
and assume that $p, 1-p\geq s_{n,m}$, where $s=s_{n,m}$ is chosen so that 
\begin{align*}
\frac{\ln^3n}{ms(1-s)}=o(1), \quad \frac{m^{6/7}}{ms(1-s)}=o(1).
\end{align*}
Then, one has
\begin{align*}
\mathbb E[\max Z_i(t)]= \sqrt{2\ln n} + \mathcal{O}(1)
\end{align*}
with a uniform error bound on all such $p$s.
\end{lemma}
The assumption in the Main Theorem guarantees that we can find a sequence
$s_{n,m}=m^{-\varepsilon'}$ for some $\varepsilon'=\varepsilon'(\varepsilon)$ sufficiently small. Therefore, we will use the notation $O_{\varepsilon}(1)$ to indicate that the error will be a function of $\varepsilon$.
\begin{proof}
We follow the notation of Theorem $1$ by Feller \cite{feller} applied to binomial random variables, which gives an uniform estimate
\begin{align*}
\mathbb P(Z_i(t)\geq x)=(1-\Phi(x))\left(1+O\left(\frac{x^3}{\sqrt{mp(1-p)}}\right)\right)
\end{align*}
where $\Phi$ denotes the cumulative distribution of a standard normal random variable.
In particular, for $x=\sqrt{2\ln n}$ the error is small by our assumption and we can write the error term as $o_{\varepsilon}(1)$. Using, for $x=\sqrt{2\ln n}$, the elementary estimate
\begin{align*}
n(1-\Phi(x))=n\frac{e^{-x^2/2}}{\sqrt{2\pi}x}\left(1+O\left(1/x^2\right)\right)\geq \frac{c}{\sqrt{\ln n}},
\end{align*}
for some absolute constant $c$ and all $n\geq 2$, we derive the bound
\begin{align*}
\mathbb P(\max Z_i(t)\leq x)&=\left(\mathbb P(Z_i(t)\leq x)\right)^n\\&=\left(1-\frac{n\mathbb P(Z_i(t)\geq x)}{n}\right)^n\\&=1-O\left(\frac{1}{\sqrt{\ln n}}(1+o_{\varepsilon}(1))\right) 
\end{align*}
From this, we get a lower bound on expectation
\begin{align*}
\mathbb E[\max Z_i(t)]&\geq x \cdot \mathbb P(\max \mathbb Z_i\leq x)=\sqrt{2\ln n}+O_{\varepsilon}(1). 
\end{align*}
As for the upper bound, we slightly extend the range and consider values of $x$ up to $x\leq \sqrt{2(\ln n+m^{1/7})}$. This range is still admissible since
$$\frac{x^3}{\sqrt{mp(1-p)}}\rightarrow 0$$
owing to our assumptions. Using the standard bound
$$ 1 - \Phi(x) \leq \frac{e^{-\frac{x^2}{2}}}{\sqrt{2\pi} x}$$
for cumulative distribution function of the Gaussian, we infer that if $x=\sqrt{2(\ln n+c)}$ for some $0\leq c\leq m^{1/7}$, then
\begin{align*}
n(1-\Phi(x))\leq \frac{ne^{-\frac{x^2}{2}}}{\sqrt{2\pi}x}=O\left(\frac{e^{-c}}{\sqrt{2\ln n}}\right)
\end{align*}
from which we obtain
\begin{align*}
n \cdot\mathbb P(Z_i(t)\geq x)= n(1-\Phi(x))(1+o_{\varepsilon}(1))=O_{\varepsilon}\left(\frac{e^{-c}}{\sqrt{\ln n}}\right).
\end{align*}
In order to bound the expectation, let $$M=\sqrt{m\max\left(\frac{1-p}{p},\frac{p}{1-p}\right)}$$ 
be the maximum value of $|Z_i|$ Notice that, for instance, $M\leq m$ owing to our assumption on $p$ and $1-p$. Then we obtain
\begin{align*}
\mathbb E[\max Z_i(t)]&\leq \int_0^{M}\mathbb P(\max Z_i(t)\geq x)dx\\&\leq \sqrt{2\ln n}+\int_{{\sqrt{2\ln n}}}^{\sqrt{2\ln n+m^{1/7}}}n\mathbb P(Z_i(t)\geq x)dx+Me^{-m^{1/7}} 
\end{align*}
The last term is certainly $o(1)$, while a change of variable shows that the second term is 
\begin{align*}
\int_{{\sqrt{2\ln n}}}^{\sqrt{2\ln n+m^{1/7}}}n\mathbb P(Z_i(t)\geq x)dx=O_{\varepsilon}\left( \int_0^{\infty}\frac{e^{-c}}{\sqrt{2\ln n+c}}dc\right)=O_{\varepsilon}\left(\frac{1}{\sqrt{\ln n}}\right).
\end{align*}
Collecting all the pieces, we have
\begin{align*}
\mathbb E[\max_i Z_i(t)]=\sqrt{2\ln n}+O_{\varepsilon}(1).
\end{align*}
\end{proof}

\subsection{Proof of the Main Result} \label{prooof}
\begin{proof}
Using Lemma \ref{fellermagic}, we have
\begin{align*}
\mathbb E\left[\max_i Y_i(t)-\frac{t}{n}\right] = \sqrt{m p(1-p)} \sqrt{2 \ln n } + \mathcal{O}(\sqrt{m p(1-p)})
\end{align*}
with error term uniform for all $p, 1-p\geq s\rightarrow 0$. Combining Lemma \ref{randomizing} with the control on the tail, we have
\begin{align*}
    S_{n,m}-m  &=   \tilde S_{n,m}-m  + \mathcal{O}(\sqrt m+\ln n)\\
    &= \sum_{t=1}^{nm}\frac{\mathbb E[\max_i Y_i(t)]}{t} - m + \mathcal{O}(\sqrt m+\ln n)\\
    &= \sum_{t= 1}^{nm}\frac{\mathbb E[\max_i \overline Y_i(t)]}{t}+ \mathcal{O}(\sqrt m+\ln n). 
\end{align*}
At this point, we split the sum into the three regions
$$ \sum_{t=1}^{mn} = \sum_{t=1}^{smn} + \sum_{t=smn}^{(1-s)mn} + \sum_{t=(1-s)mn}^{mn}.$$
As was shown in Section \S 2.3, as long as $s \rightarrow 0$, the first and the third sum
are $o(\sqrt{m \ln{n}})$. The sum in the middle, provided $s$ does not tend to 0 too quickly,
will then contribute
$$ \sum_{t= smn}^{(1-s)nm}\frac{\mathbb E[\max_i \overline Y_i(t)]}{t} = \left(1+\mathcal{O}\left( \frac{1}{\sqrt{\ln{n}}}\right)\right)\sum_{t= smn}^{(1-s)nm}\frac{ \sqrt{m p(1-p)} \sqrt{2 \ln n } }{t}.$$
The sum can now be simplified to
$$ \sum_{t= smn}^{(1-s)nm}\frac{ \sqrt{m p(1-p)} \sqrt{2 \ln n } }{t} = \sqrt{2 m \ln{n}} \sum_{t= smn}^{(1-s)nm}\frac{ \sqrt{ p(1-p)}  }{t} $$
which, recalling $p = t/mn$ leads, as $s \rightarrow 0$, to the Riemann sum
$$\int_0^1\sqrt{\frac{1-p}{p}}dp = \frac{\pi}{2}.$$
\end{proof}
\section*{Acknowledgment}
We warmly thank Persi Diaconis for suggesting the problem and for some useful references.
 \bibliographystyle{abbrv}
  \bibliography{Finalversion.bib}

\def\cprime{$'$}
\begin{thebibliography}{10}

\bibitem{Blackwell1957}
D.~Blackwell and J.~L. Hodges.
\newblock Design for the control of selection bias.
\newblock {\em The Annals of Mathematical Statistics}, 28(2):449--460, 1957.

\bibitem{C98}
M.~Ciucu.
\newblock No-feedback card guessing for dovetail shuffles.
\newblock {\em Ann. Appl. Probab.}, 8(4):1251--1269, 1998.

\bibitem{renewal}
D.~J. Daley.
\newblock Another upper bound for the renewal function.
\newblock {\em The Annals of Probability}, 4(1):109--114, 1976.

\bibitem{Diaconis1978}
P.~Diaconis.
\newblock Statistical problems in {ESP} research.
\newblock {\em Science}, 201(4351):131--136, 1978.

\bibitem{DG81}
P.~Diaconis and R.~Graham.
\newblock The analysis of sequential experiments with feedback to subjects.
\newblock {\em Ann. Statist.}, 9(1):3--23, 1981.

\bibitem{diaconis2020card}
P.~Diaconis, R.~Graham, X.~He, and S.~Spiro.
\newblock Card guessing with partial feedback.
\newblock {\em Combinatorics, Probability and Computing}, page 1^^e2^^80^^9320,
  2021.

\bibitem{diaconis2020guessing}
P.~Diaconis, R.~Graham, and S.~Spiro.
\newblock Guessing about guessing: Practical strategies for card guessing with
  feedback.
\newblock {\em arXiv preprint arXiv:2012.04019}, 2020.

\bibitem{EFRON1971}
B.~Efron.
\newblock Forcing a sequential experiment to be balanced.
\newblock {\em Biometrika}, 58(3):403--417, 1971.

\bibitem{feller}
W.~Feller.
\newblock Generalization of a probability limit theorem of cramer.
\newblock {\em Transactions of the American Mathematical Society},
  54(3):361--372, 1943.

\bibitem{he2021card}
J.~He and A.~Ottolini.
\newblock Card guessing and the birthday problem for sampling without
  replacement.
\newblock {\em arXiv preprint arXiv:2108.07355}, 2021.

\bibitem{KP01}
A.~Knopfmacher and H.~Prodinger.
\newblock A simple card guessing game revisited.
\newblock {\em Electron. J. Combin.}, 8(2):Research Paper 13, 9, 2001.
\newblock In honor of Aviezri Fraenkel on the occasion of his 70th birthday.

\bibitem{KT21}
T.~Krityakierne and T.~A. Thanatipanonda.
\newblock The card guessing game: A generating function approach.
\newblock {\em arXiv preprint arXiv:2107.11142}, 2021.

\bibitem{KPP09}
M.~Kuba, A.~Panholzer, and H.~Prodinger.
\newblock Lattice paths, sampling without replacement, and limiting
  distributions.
\newblock {\em Electron. J. Combin.}, 16(1):Research Paper 67, 12, 2009.

\bibitem{Levin1981}
B.~Levin.
\newblock A representation for multinomial cumulative distribution functions.
\newblock {\em The Annals of Statistics}, 9(5), Sept. 1981.

\bibitem{L21}
P.~Liu.
\newblock On card guessing game with one time riffle shuffle and complete
  feedback.
\newblock {\em Discrete Appl. Math.}, 288:270--278, 2021.

\bibitem{P09}
L.~Pehlivan.
\newblock {\em On top to random shuffles, no feedback card guessing, and fixed
  points of permutations}.
\newblock ProQuest LLC, Ann Arbor, MI, 2009.
\newblock Thesis (Ph.D.)--University of Southern California.

\bibitem{P91}
M.~Proschan.
\newblock A note on {B}lackwell and {H}odges (1957) and {D}iaconis and {G}raham
  (1981).
\newblock {\em Ann. Statist.}, 19(2):1106--1108, 1991.

\bibitem{Skibinsky1970}
M.~Skibinsky.
\newblock A characterization of hypergeometric distributions.
\newblock {\em Journal of the American Statistical Association},
  65(330):926--929, June 1970.

\bibitem{S21}
S.~Spiro.
\newblock Online card games.
\newblock {\em arXiv preprint arXiv:2106.11866}, 2021.

\end{thebibliography}

\end{document}